\documentclass[12pt]{amsart}
\usepackage{amscd, graphicx}
\usepackage{tikz}
\def\NZQ{\mathbb}               
\def\NN{{\NZQ N}}

\def\RR{{\NZQ R}}

\def\mm{{\mathfrak m}}

\def\a{{\mathbf a}}
\def\b{{\mathbf b}}

\def\1{{\mathbf 1}}
\def\0{{\mathbf 0}}
\def\opn#1#2{\def#1{\operatorname{#2}}} 
\opn\diam{diam} 
\opn\depth{depth}
\opn\ini{in}
\opn\gcd{gcd}

\newtheorem{Theorem}{Theorem}[section]
\newtheorem{Lemma}[Theorem]{Lemma}
\newtheorem{Corollary}[Theorem]{Corollary}
\newtheorem{Proposition}[Theorem]{Proposition}
\newtheorem{Remark}[Theorem]{Remark}

\begin{document} 

\title{The Bhattacharya function of \\ complete monomial ideals  
in two variables}
\author{Hong Ngoc Binh}
\address{Institute of Mathematics, Vietnam Academy of Science and Technology, 18 Hoang Quoc Viet, Hanoi, Vietnam}
\email{binh.hongngoc@gmail.com}

\author{Ngo Viet Trung}
\address{Institute of Mathematics, Vietnam Academy of Science and Technology, 18 Hoang Quoc Viet, Hanoi, Vietnam}
\email{nvtrung@math.ac.vn}


\begin{abstract}
We give explicit formulas for the Bhattacharya function of $\mm$-primary complete monomial ideals in two variables in terms of the vertices of the Newton polyhedra or in terms of the decompositions of the ideals as products of simple ideals.
\end{abstract}

\thanks{The second author is supported by the National Foundation of  Science and Technology Development.} 
\maketitle

\section*{Introduction}

Let $(R,\mm)$ be a $d$-dimensional local ring. For two $\mm$-primary ideals $I,J$ of $R$, 
one can consider the length function $\ell(R/I^mJ^n)$, $m,n \ge 0$.
This function was studied first by Bhattacharya \cite{Bh} who proved that this function is a polynomial $P(m,n)$ of total degree $d$ 
for $m, n \gg 0$. One calls $\ell(R/I^mJ^n)$ and $P(m,n)$ the Bhattacharya function and the  Bhattacharya polynomial of $I,J$. 
If we write
$$P(m,n) = \sum_{0 \le i,j \le d} e_{i,j}{m \choose i}{n \choose j},$$
then the numbers $e_{i,j}$ with $i + j = d$ are called the mixed multiplicities of $I,J$.
In general, it is very difficult to compute the Bhattacharya polynomial or the mixed multiplicities of a given pair of ideals $I,J$. \par

If $I$ and $J$ are $\mm$-primary monomial ideals of a polynomial ring $R$ over a field, where $\mm$ is the maximal homogeneous ideal,
Teissier \cite[Appendix]{Te} showed that the mixed multiplicities can be interpreted in terms of the volumes of the complements of the Newton polyhedra of $I$, and $J$. Recall that the Newton polyhedron of a monomial ideal is the convex hull of the exponent vectors of the monomials of the ideal. But nothing is known about the other coefficients of the Bhattacharya polynomial in this general setting. 
\par

The only case where one can compute the Bhattacharya polynomial is the case $(R,\mm)$ is a two-dimensional regular local ring and $I,J$ are $\mm$-primary complete ideals. Recall that an ideal is called complete if it is integrally closed. Using joint reductions of complete ideals and Lipman's formula for mixed multiplicities, Verma \cite{Ve} showed that the Bhattacharya function coincides with the Bhattacharya polynomial and that its coefficients  can be expressed in terms of the multiplicities and the lengths of $R/I$, $R/J$, and $R/IJ$. \par

In this paper we will compute the Bhattacharya polynomial in the case $R$ is a polynomial ring in two variables and $I,J$ are $\mm$-primary complete monomial ideals. Our results show that the coefficients of the Bhattacharya polynomial can be explicitly expressed in terms of the vertices of Newton polyhedra of $I$ and $J$ or in terms of the decompositions of $I,J$ as products of simple ideals. Recall that a complete ideals is called simple if it is not the product of two proper complete ideals. \par

Our approach is based on a result of Crispin Quinonez \cite{Cr} who described the decomposition of an $\mm$-primary complete monomial ideal as a product of simple ideals in terms of the vertices of the Newton polyhedron. The existence of such a decomposition follows from Zariski's theory of complete ideals in a two-dimensional regular local ring \cite[Appendix 5]{Za}. The description of such a decomposition allows us to study the combinatorial properties of the product of two complete monomial ideals. 
We will reduce the problem of computing the Bhattacharya function to the problem of counting lattice points of the complement of the Newton polyhedron and we will use Pick's theorem to relate the number of lattice points to the area of the corresponding polygon.
If we know the decomposition of the ideal as a product of simple ideals, we can compute the area of such a polygon by the theory of Minkowski sum and mixed areas.  \par

Our results do not have any theoretical contribution to the theory of Bhattacharya function and mixed multiplicities.
However, we feel that combinatorial formulas for the Bhattacharya function of complete monomial ideals in two variables would be useful for the study of mixed multiplicities and similar functions of monomial ideals in several variables. 
For instance, if $I$ is an $\mm$-primary monomial ideal and $J$ an arbitrary monomial ideal, one can use the function $\ell(I^nJ^m/I^{n+1}J^m)$ to define mixed multiplicities of $I,J$ \cite{Bh}, \cite{KaV}, \cite{TrV2}. So far, no combinatorial formula for these mixed multiplicities is known except the case $I = \mm$ and $J$ is generated by monomials of the same degree \cite{TrV}. Our results may give some hints for such a formula
since the computation of the function $\ell(I^nJ^m/I^{n+1}J^m)$ can be reduced to the computation of a Bhattacharya function in the two variables case.
\par

The paper is organized as follows. In Section 1 we describe the decomposition of complete monomial ideals.
In Section 2 we show how to compute the colength of an $\mm$-primary complete monomial ideal. 
Explicit formulas for the Bhattacharya function and their consequences are given in Section 3.\par

\noindent{\em Acknowledgement}. 
The authors would like to thank the referee for his suggestions which have helped to improve the presesentation of the paper.

\section{Decomposition of complete ideals}

Let $R$ be a Noetherian ring and $I$ an ideal of $R$.
An element $x \in R$ is said to be {\em integral} over $I$ if it satisfies an equation of the form
$$x^n + c_1x^{n-1} + \cdots + c_n = 0$$
where $c_j \in I^j$ for $j = 1,...,n$.
It is known that the set of all integral elements over $I$ form an ideal of $R$.
This ideal is called the {\em integral closure} of $I$, denoted by $\overline I$.
If $\overline I = I$, then $I$ is called an {\em integrally closed} or {\em complete} ideal. 
For more information on integrally closures of ideals we refer to the books \cite{HuS} and \cite{Va}. \par

The complete ideals behave especially well if $R$ is a two-dimensional regular local ring.
In this case we know by Zariski \cite[Appendix 5]{Za} 
that the product of two complete ideals is complete and  
every complete ideal can be uniquely written as a product of simple complete ideals,
where a complete ideal is {\em simple} if it is not the product of two complete ideals.
\par

If $I$ is a monomial ideal in a polynomial ring $R = k[x_1,...,x_n]$, 
the integral closure $\overline I$ can be described combinatorially as follows.
Let $N(I)$ denote the {\em Newton polyhedron} of $I$, that is the convex hull of the exponent vectors of the monomials of $I$ in $\RR^n$.
Then $\overline I$ is the ideal generated by all monomials whose exponent vectors belong to $N(I)$ (see \cite{HuS} or \cite{Va}).
Let $B(I)$ denote the union of the compact faces of $N(I)$.
If we define a partial order on $\RR^n$ by the rule $\a \le \b$ if each component of $\a$ is less or equal than the corresponding component of $\b$, then $N(I)$ is the set of all points which are greater or equal the points of $B(I)$.
So $I$ is completely determined by $B(I)$.
We call $B(I)$ the {\em Newton boundary} of $I$. 

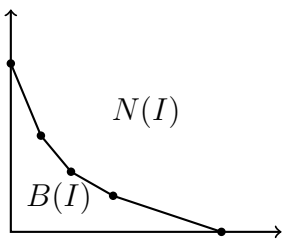
\begin{figure}[ht!]
\begin{tikzpicture}[scale=0.8]
\draw [<->,thick] (0,3.7) node (yaxis) [above] {}
        |- (4.5,0) node (xaxis) [right] {};
    
\draw [thick] (0,2.8) coordinate (a) -- (0.5,1.6) coordinate (b) ;
\draw [thick] (0.5,1.6) coordinate (b) -- (1,1) coordinate (c) ;
\draw [thick] (1,1) coordinate (c) -- (1.7,0.6) coordinate (d) ; 
\draw [thick] (1.7,0.6) coordinate (d) -- (3.5,0) coordinate (e);

\draw (0.8,1) node[below] {$B(I)$};
\draw (1.5,2) node[above,right] {$N(I)$};
     
\fill (a) circle (2pt);
    \fill (b) circle (2pt);
  \fill (c) circle (2pt);
  \fill (d) circle (2pt);
  \fill (e) circle (2pt);
\end{tikzpicture}
 \end{figure}

If $R = k[x,y]$ is a polynomial ring in two variables, the homogeneous version of Zariski's decomposition  theorem implies that 
the product of two complete homogeneous ideals in $R$ is complete and that
every complete homogeneous ideal can be uniquely written as a product of simple homogeneous ideals. 
Let $I$ be a complete monomial ideal in $R = k[x,y]$. 
One may ask whether there is a description of the simple homogeneous ideals of $I$ in terms of $B(I)$.
\par

The answer is yes and is due to Crispin Quinonez \cite{Cr}. 
She calls the integral closure of a complete intersection ideal $(x^p,y^q)$ with $\gcd(p,q) = 1$ a {\em block ideal}
and proved that a block ideal is not the product of two complete monomial ideals \cite[Proposition 3.4]{Cr}.
Actually, one can prove more.

\begin{Proposition} \label{simple} 
A block ideal is simple.
\end{Proposition}

\begin{proof}
Let $I= \overline{(x^p,y^q)}$ with $\gcd(p,q) = 1$.
Then $B(I)$ is the line segment connecting the points $(p,0)$ and $(0,q)$ and they are the only lattice points of $B(I)$.
Since $mq + np = pq$ is the equation of the supporting line of this segment, 
all lattice points $(m,n) \in N(I)$ satisfy the condition $mq + np \ge pq$. 
Consider $R$ as a weighted graded ring with $\deg x = q, \deg y = p$.
By the above description of $N(I)$, all monomials of $I$ have degree $\ge pq$ 
and $x^p, y^q$ are the only monomials having degree $pq$. \par

Assume that $I$ is the product of two complete ideals $I_1,I_2$.
For every polynomial $f$ we denote by $\ini(f)$ the initial monomial of $f$ with respect to the weighted degree.
Let $f_1 \in I_1$ and $f_2 \in I_2$ such that $\deg(\ini(f_1))$ resp. $\deg(\ini(f_2))$ are minimal among 
the degree of the initial terms of the polynomials of $I_1$ resp. $I_2$.
Since $f_1f_2 \in I$ and $I$ is a monomial ideal, $\ini(f_1)\ini(f_2) \in I$ and $\deg(\ini(f_1)\ini(f_2)) = pq$, 
the smallest degree of the monomials of $I$.
Therefore, $\ini(f_1)\ini(f_2)$ equals $x^p$ or $y^q$.
Without restriction we may assume that $\ini(f_1)\ini(f_2) = x^p$.
Then $\ini(f_1) = x^a$ and $\ini(f_2) = x^b$ for some fixed non-negative integers $a,b$ with $a+b =p$.
From this it follows that  for all $g_1 \in I_1$ resp. $g_2 \in I_2$, we either have $\ini(g_1) = x^a$ or $\deg(\ini(g_1)) > \deg(\ini(f_1)) = aq$ resp. $\ini(g_2) = x^b$ or $\deg(\ini(g_2)) > \deg(\ini(f_2)) = bq$. 
Thus, $\ini(g_1g_2) = x^ax^b = x^p$ or $\deg(\ini(g_1g_2)) > \deg(\ini(f_1)) + \deg(\ini(f_2)) = pq$.
Since $\deg(y^q) = pq$, this implies $y^q \not\in I_1I_2 = I$, a contradiction.
\end{proof} 

Let $\mm$ be the maximal homogeneous ideal of $R = k[x,y]$. 
Since every monomial ideal in $R$ is the product of a monomial with an $\mm$-primary monomial ideal,
we may assume that $I$ is $\mm$-primary. 
In this case, Crispin Quinonez proved that $I$ can be uniquely decomposed as a product of block ideals.

Her proof also describes the block ideals of $I$ from the vertices of the Newton boundary $B(I)$. 
Let $(a_1,b_1),...,(a_{r+1},b_{r+1})$ be the vertices of $N(I)$. 
Without loss of generality we may assume that  $0 = a_1 < \cdots < a_{r+1}$, which implies $b_1 > \cdots > b_{r+1} = 0$.
Let $L_i$ denote the line segment connecting $(a_i,b_i)$ to $(a_{i+1},b_{i+1})$, $i = 1,...,r$. 
Then $B(I) = L_1 \cup \cdots \cup L_r$.

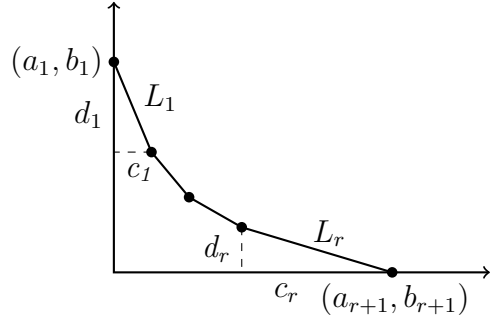
\begin{figure}[ht!]
\begin{tikzpicture}[scale=1]
\draw [<->,thick] (0,3.6) node (yaxis) [above] {}
        |- (5,0) node (xaxis) [right] {};
    
\draw [thick] (0,2.8) coordinate (a) -- (0.5,1.6) coordinate (b) ;
\draw [thick] (0.5,1.6) coordinate (b) -- (1,1) coordinate (c) ;
\draw [thick] (1,1) coordinate (c) -- (1.7,0.6) coordinate (d) ; 
\draw [thick] (1.7,0.6) coordinate (d) -- (3.7,0) coordinate (e);

 \draw (0,2.8) node[left] {$(a_1,b_1)$};
\draw (0,2.1) node[left] {$d_1$};
 \draw (0.25,2) node[above right] {$L_1$};
 \draw (2.5,0.2) node[above right] {$L_r$};
 \draw (2.6,0) node[below right] {$\small{ (a_{r+1},b_{r+1})}$};
 \draw (2.3,0) node[below] {$c_r$};

 \draw[dashed] 
 (b)--(yaxis |- b)   node[below right] {$\mathit{c_1}$}
     (d)--(xaxis -| d) node[above left] {$\mathit{d_r}$};
    
\fill (a) circle (2pt);
    \fill (b) circle (2pt);
  \fill (c) circle (2pt);
  \fill (d) circle (2pt);
  \fill (e) circle (2pt);
\end{tikzpicture}
 \end{figure}

Put $c_i = a_{i+1}-a_i$ and $d_i = b_{i}-b_{i+1}$, $i = 1,...,r$.
Let $p_i, q_i$ be positive numbers with $\gcd(p_i,q_i) = 1$ such  that $p_i/q_i = d_i/c_i$. 
Let $C_i = \overline{(x^{p_i},y^{q_i})}$ and $n_i = \gcd(c_i,d_i)$.

\begin{Theorem}  \label{decomposition}  {\rm (\cite[Theorem 3.8]{Cr})}
Let $I$ be an $\mm$-primary complete monomial ideal in $R$. With the above notations we have
$$I = C_1^{n_1} \cdots C_r^{n_r}.$$
\end{Theorem} 

Geometrically, the slope of $L_i$ is given by the ratio $d_i/c_i$ and $n_i+1$ is the number of lattice points on $L_i$.
So we can decompose $L_i$ into $n_i$ line segments whose interior does not contain any lattice point.
Since $B(C_i)$ is the line segment connecting $(p_i,0)$ to $(0,q_i)$ which has the same slope,
we may consider $L_i$ as the union of $n_i$ copy of $B(C_i)$.
So we can read off the block ideals of $I$ from $B(I)$. \par

\begin{Remark} \label{option}
{\rm It is sometime more convenient to write $I$ as a product of integral closures of complete intersections 
which are not necessarily block ideals. Let $I = J_1 \cdots J_r$, where $J_i = \overline{(x^{c_i},y^{d_i})}$
for some positive integers  $c_i,d_i$, $i = 1,...,r$. Define $p_i, q_i$, and $n_i$ as above. Then $I = C_1^{n_1} \cdots C_r^{n_r}$ is a product of block ideals, where $C_1,...,C_r$ are not necessarily different. Without restriction we may assume that $d_1/c_1 \ge \cdots \ge d_r/c_r$. Define 
$a_1 = 0$, $b_{r+1} = 0$ and
\begin{align*}
a_i & = c_1 + \cdots + c_{i-1},\ i = 2,...,r+1,\\
b_i & = d_i + \cdots + d_r,\ i = 1,...,r.
\end{align*}
Then we still have $B(I) = L_1 \cup \cdots \cup L_r$, although the points $(a_1,b_1),...,(a_{r+1},b_{r+1})$ need not to be the vertices of $N(I)$.}
\end{Remark}

For any subset $P \subset \RR_+^2$ we denote by $l_P$ the number of lattice points in $P$.
For short we set $l_I = l_{B(I)}$. 

\begin{Lemma}  \label{boundary}  
Let $I = J_1 \cdots J_r$, where $J_i = \overline{(x^{c_i},y^{d_i})}$
for some positive integers  $c_i,d_i$, $i = 1,...,r$. Then 
$$l_I = \sum_{i=1}^r \gcd(c_i,d_i) + 1,$$
and $l_I - 1$ is the number of block ideals in the decomposition of $I$.
\end{Lemma}

\begin{proof}
By Remark \ref{option}, $n_1+ \cdots +n_r$ is the number of block ideals in the decomposition of $I$.
On the other hand, $n_i+1$ is the number of lattice points on $L_i$, $i = 1,...,r$.
Hence $n_1+ \cdots +n_r +r$ is the sum of the numbers of lattice points on $L_1,...,L_r$.
Since $B(I) = L_1 \cup \cdots \cup L_r$ and since the points $(a_2,b_2),...,(a_r,b_r)$ are counted twice in the above sum, we get
$$l_I = (n_1+ \cdots + n_r + r) - (r-1) = n_1 + \cdots + n_r +1 = \sum_{i=1}^r \gcd(c_i,d_i) + 1.$$  
\end{proof} 

\begin{Lemma}  \label{product number}  
Let $I = J_1 \cdots J_r$, where $J_1,...,J_r$ are $\mm$-primary complete monomial ideals in $R$.  Then
$$l_{J_1 \cdots J_r} = l_{J_1} + \cdots + l_{J_r} - r +1.$$
\end{Lemma}

\begin{proof}
By Theorem \ref{decomposition}, the number of block ideals in the decomposition of $J_1 \cdots J_r$ is the sum 
of the numbers of block ideals in the decompositions of $J_1,...,J_r$.
Hence the assertion follows from Lemma \ref{boundary}.
\end{proof}

\section{Colength of an complete monomial ideal}

Let $R$ be a polynomial ring over a field $k$.
Let $I$ be an $\mm$-primary complete monomial ideal in $R$, 
where $\mm$ is the maximal homogeneous ideal of $R$.
Then $R/I$ is an $R$-module of finite length. One calls $\ell(R/I)$ the {\em colength} of $I$.
It is clear that $\ell(R/I)$ is the number of the monomials not in $I$. 
Since $I$ is generated by the monomials whose exponent vectors belong to $N(I)$, 
the monomials not in $I$ correspond to the lattice points of the complement of $N(I)$ in $\NN^n$. \par

Let $I$ be an $\mm$-primary complete monomial ideal in $R = k[x,y]$.  
Let $(a_I,0), (0,b_I)$ be the end points of the Newton boundary $B(I)$. 
Let $\overline{P(I)}$ be the polygon confined by $B(I)$ and the two line segments connecting the origin $(0,0)$ with the points $(a_I,0), (0,b_I)$. Then $\overline{P(I)} \setminus B(I)$ is the complement of $N(I)$ in $\NN^n$. Hence
$\ell(R/I) = l_{\overline{P(I)}} - l_I.$

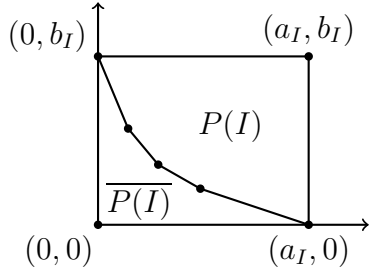
\begin{figure}[ht!]
\begin{tikzpicture}[scale= 0.8]
\draw [<->,thick] (0,3.7) node (yaxis) [above] {}
        |- (4.5,0) node (xaxis) [right] {};
    
\draw [thick] (0,2.8) coordinate (a) -- (0.5,1.6) coordinate (b) ;
\draw [thick] (0.5,1.6) coordinate (b) -- (1,1) coordinate (c) ;
\draw [thick] (1,1) coordinate (c) -- (1.7,0.6) coordinate (d) ; 
\draw [thick] (1.7,0.6) coordinate (d) -- (3.5,0) coordinate (e);
\draw  [thick] (3.5,0)  -- (3.5,2.8) coordinate (f);
\draw [thick] (0,2.8)  -- (3.5,2.8);

 \draw (0,2.7) node[above left] {$(0, b_I)$};
\draw (3.5,0) node[below] {$(a_I, 0)$};
\draw (3.5,2.8) node[above] {$(a_I, b_I)$};
 \draw (2.2,1.2) node[above] {$P(I)$};
  \draw (0.1,0) node[below left] {$(0,0)$};
\draw (1.4,0.9) node[below left] {$\overline{P(I)}$};
     
\fill (0,0) circle (2pt);
\fill (a) circle (2pt);
    \fill (b) circle (2pt);
  \fill (c) circle (2pt);
  \fill (d) circle (2pt);
  \fill (e) circle (2pt);
  \fill (f) circle (2pt);
\end{tikzpicture}
 \end{figure}
 
We can estimate $l_{\overline{P(I)}}$ by using Pick's theorem which relates the lattice points of a polygon with its area. 
Recall that a polygon is a {\em lattice polygon} if its vertices are lattice points.  
\smallskip

Let $V(P)$ denote the area of a lattice polygon $P$. 
Let $i_P$ resp. $b_P$ be the numbers of lattice points in the interior  resp. the boundary of $P$. 

\begin{Theorem}  \label{Pick}  {\rm (Pick's Theorem, see e.g. \cite{Gr})}
$\displaystyle V(P) = i_P + \frac{b_P}{2} - 1.$
\end{Theorem}

In the following we set $s_I := V(\overline{P(I)})$. 

\begin{Lemma}  \label{colength}
$\displaystyle \ell(R/I) = s_I + \frac{1}{2}(a_I+b_I-l_I+1).$
\end{Lemma}

\begin{proof}
By Theorem \ref{Pick} we have  
$$l_{\overline{N(I)}} = l_{\overline{P(I)}} - l_I = i_{\overline{P(I)}} + b_{\overline{P(I)}}- l_I = s_I  +  \frac{b_{\overline{P(I)}}}{2} +1 - l_I.$$
It is easy to see that $b_{\overline{P(I)}} = a_I + b_I + l_I -1$. Therefore,
$$\ell(R/I)  =  s_I + \frac{1}{2}(a_I+b_I+l_I-1) +1 - l_I = s_I + \frac{1}{2}(a_I+b_I-l_I+1).$$
\end{proof} 

\begin{Remark} \label{convex}
{\rm In general, $\overline{P(I)}$ is not a convex polygon. 
However, one can reduce the computation of $s_I$ to the computation of volumes of  convex polygons.
Let $Q(I)$ denote the rectangle with the vertices $(0,0),(a_I,0), (b_I,0),(a_I,b_I)$.
Let $P(I)$ denote the convex polygon defined by $B(I)$ and the two line segments connecting $(a_I,0), (0,b_I)$ with $(a_I,b_I)$. Then}
\begin{equation*}
s_I = V(Q(I)) - V(P(I)). \tag{*} 
\end{equation*}
\end{Remark}

If we know a decomposition of $I$ as a product of integral closures of complete intersections, we can compute $s_I$ directly.

\begin{Lemma} \label{area}
Let $I = J_1  \cdots J_r$, where $J_i = \overline{(x^{c_i},y^{d_i})}$ for some positive integers $c_i, d_i$, $i = 1,...,r$. Assume that $d_1/c_1 \ge \cdots \ge d_r/c_r$.  Then
$$s_I = \sum_{i=1}^r\frac{c_id_i}{2} + \sum_{1 \le i < j \le r}c_id_j.$$
\end{Lemma}

\begin{proof}
We proceed by induction.
If $r = 1$, the assertion is trivial.
If $r > 1$, let $J = J_1 \cdots J_{r-1}$. By the induction hypothesis,
$$s_J = \sum_{i=1}^{r-1}\frac{c_id_i}{2} + \sum_{1 \le i < j \le r-1}c_id_j.$$
Define $a_i, b_i$ as in Remark \ref{option}, $i =1,...,r+1$. Then $B(I) = L_1 \cup \cdots \cup L_r$, where $L_i$ is the line segment connecting $(a_i,b_i)$ to $(a_{i+1},b_{i+1})$, $i = 1,...,r$. It is clear that $\overline{P(J)}$ is a translation of the polygon $P$ confined by $B(I)$, the vertical axis, and the horizontal line passing through the point $(a_r,b_r)$: 
\begin{figure}[ht!]
\begin{tikzpicture}[scale=1]
\draw [<->,thick] (0,3.6) node (yaxis) [above] {}
        |- (5.2,0) node (xaxis) [right] {};
    
\draw [thick] (0,2.8) coordinate (a) -- (0.5,1.6) coordinate (b) ;
\draw [thick] (0.5,1.6) coordinate (b) -- (1,1) coordinate (c) ;
\draw [thick] (1,1) coordinate (c) -- (1.6,0.7) coordinate (d) ; 
\draw [thick] (1.6,0.7) coordinate (d) -- (4,0) coordinate (e);
\draw [dashed] (1.6,0.7) coordinate (d) -- (0,0.7);

 \draw (0,2.8) node[left] {$(a_1,b_1)$};
 \draw (0.25,2) node[above right] {$L_1$};
   \draw (0.4,0.8) node[above] {$P$};
 \draw (2.7,0.2) node[above right] {$L_r$};
  \draw (2,0.7) node[above] {$(a_r,b_r)$};
 \draw (2.9,0) node[below right] {$\small{ (a_{r+1},b_{r+1})}$};
 \draw (2.4,0) node[below] {$c_r$};

 \draw[dashed] 
(d)--(xaxis -| d) node[above left] {$\mathit{d_r}$};
    
\fill (a) circle (2pt);
    \fill (b) circle (2pt);
  \fill (c) circle (2pt);
  \fill (d) circle (2pt);
  \fill (e) circle (2pt);
\end{tikzpicture}
 \end{figure}

Thus,  $s_I - s_J = V(\overline{P(I)} \setminus P) = \displaystyle \frac{c_rd_r}{2} +(c_1+\cdots+c_{r-1})d_r$. Hence the assertion is immediate. 
\end{proof} 

Now we can deduce an explicit formula for the colength of $I$ in terms of the decomposition of $I$ as a product of  block ideals.

\begin{Theorem}    \label{length}
Let $I = C_1^{n_1}  \cdots C_r^{n_r}$ be the decomposition of $I$ as a product of  block ideals
$C_i = \overline{(x^{p_i},y^{q_i})}$, $i = 1,...,r$.
Assume that $q_1/p_1 > \cdots > q_r/p_r$. Then
$$\ell(R/I) = \sum_{i=1}^r\frac{p_iq_i}{2}n_i^2 + \sum_{1 \le i < j \le r}p_iq_j n_in_j+ \frac{1}{2}\sum_{i=1}^r(p_i+q_i-1)n_i.$$
\end{Theorem}

\begin{proof}
Let $J_i = \overline{(x^{p_in_i},y^{q_in_i})}$, $i = 1,...,r$. 
Then $J_i = C_i^{n_i}$. Hence $I = J_1 \cdots J_r.$ 
By Lemma \ref{area} we have
$$s_I = \sum_{i=1}^r\frac{p_iq_i}{2}n_i^2 + \sum_{1 \le i < j \le r}p_iq_jn_in_j.$$
It is clear that $a_I = p_1n_1 + \cdots + p_rn_r$ and $b_I = q_1n_1+ \cdots + q_rn_r$.
By Lemma \ref{boundary}, $l_I = n_1 + \cdots + n_r +1$. 
If we put these data into  Lemma \ref{colength}, we obtain the assertion. 
\end{proof}

\section{Bhattacharya function}

Let $(R,\mm)$ be a $d$-dimensional local ring.
Let $I,J$ be two $\mm$-primary ideals. Recall that the function $\ell(R/I^mJ^n)$  is a polynomial $P(m,n)$ 
of total degree $d$ for $m, n \gg 0$ and that $\ell(R/I^mJ^n)$ and $P(m,n)$ are called 
the {\em Bhattacharya function} and the {\em Bhattacharya polynomial} of $I,J$. 
If we write
$$P(m,n) = \sum_{0 \le i,j \le d} e_{i,j}{m \choose i}{n \choose j},$$
the numbers $e_{i,j}$ with $i + j = d$ are called the {\em mixed multiplicities} of $I,J$. \par

In this section we will deal with the case $R = k[x,y]$ is a polynomial ring of two variables over a field $k$ and $I,J$ are two $\mm$-primary complete monomial ideals, where $\mm$ is the maximal homogeneous ideal.
We will give an explicit formula for $\ell(R/I^mJ^n)$ in combinatorial terms of $I$ and $J$. 
For that we shall need the following notations. \par

Given two (not necessarily different) convex polygons $P_1$ and $P_2$, the {\em Minkowski sum} of $P_1, P_2$ is defined by
$$P_1 + P_2 = \{\a+\b|\ \a \in P_1, \b \in P_2\}.$$
It is easy to see that $P_1+P_2$ is again a convex polygon. 
The {\em mixed area} $MV(P_1,P_2)$ of $P_1$ and $P_2$ is defined by
$$MV(P_1,P_2) :=  V(P_1+P_2) - V(P_1) - V(P_2).$$
Let $mP_1$ and $nP_2$ denote the Minkowski sums of $n$ times $P_1$ and $n$ times $P_2$, respectively.
It is a classical result that the area of $mP_1 + nP_2$ is a homogeneous polynomial in $m,n$ which involves the mixed area
of $P_1$ and $P_2$.
 
\begin{Lemma}  \label{mixed} {\rm (see e.g. \cite[Ch. 7, Proposition 4.9]{Co})}
Let $P_1$ and $P_2$ be convex polygons in $\RR^2$. Then
$$V(mP_1 + nP_2) =  V(P_1)m^2 + V(P_2)n^2 + MV(P_1,P_2)mn.$$
\end{Lemma}

\begin{Lemma}  \label{additive}
Let $I$ and $J$ be  $\mm$-primary complete monomial ideals in $R$. Then\par
{\rm (i) } $Q(I^mJ^n) = mQ(I) + nQ(J),$ \par
{\rm (ii)} $P(I^mJ^n)  =  mP(I) + nP(J).$
\end{Lemma}

\begin{proof} 
We only need to prove the case $m = n = 1$ because this case can be applied successively to obtain the assertion.
The formula  $Q(IJ) = Q(I) + Q(J)$ follows from the definition of the Minkowski sum and the facts that $a_{IJ} = a_I +a_J$ and $b_{IJ} = b_I + b_J$. 
By the definition of the Newton polyhedron we have $N(IJ) = N(I) + N(J)$. 
Since the boundary of a Minkowski sum is contained in the Minkowski sum of the boundary of the summands, 
$B(IJ) \subseteq B(I) + B(J)$. From this it follows that $P(IJ) \subseteq P(I) + P(J)$.
On the other hand, 
$$P(I) + P(J) \subseteq (Q(I) + Q(J)) \cap (N(I) + N(J)) = Q(IJ) \cap N(IJ) = P(IJ).$$
Therefore, $P(IJ) =  P(I)  +  P(J).$
\end{proof}

\begin{Theorem}  \label{Bhattacharya}  
Let $I, J$ be $\mm$-primary complete monomial ideals in $R = k[x,y]$.  For all $m,n \ge 0$,
\begin{align*}
\ell(R/I^mJ^n) = &\ s_Im^2 + s_Jn^2 + (s_{IJ} - s_I - s_J)mn\\
&  + \frac{1}{2}(a_I +b_I-l_I+1)m +  \frac{1}{2}(a_J +b_J-l_J+1)n.
\end{align*}
\end{Theorem}

\begin{proof}
We will use Lemma \ref{colength} to compute $\ell(R/I^mJ^n)$. 
By Remark \ref{convex} (*)  we have  
$$s_{I^mJ^n} = V(Q(I^mJ^n)) - V(P(I^mJ^n)). $$
Using Lemma \ref{mixed} and Lemma \ref{additive} we obtain a formula for $V(Q(I^mJ^n)) - V(P(I^mJ^n))$ in terms of the areas and mixed areas of $Q(I), Q(J)$ and $P(I), P(J)$.  From this formula and Remark \ref{option} (*) (applied to the ideals $IJ, I, J$) we can deduce that
$$s_{I^mJ^n} =  s_Im^2 + s_Jn^2 + (s_{IJ} - s_I - s_J)mn.$$
It is clear that $a_{I^mJ^n} =  a_Im + a_Jn$ and $b_{I^mJ^n} =  b_Im + b_Jn$.
By Lemma \ref{product number}, 
$l_{I^mJ^n} = (l_J-1)m + (l_I-1)n + 1$.  Hence the assertion follows from Lemma \ref{colength}.
\end{proof} 

From Theorem \ref{Bhattacharya} we obtain the following formulas for the mixed multiplicities of $I,J$: 
$e_{2,0} =  2s_I,\ e_{0,2} = 2s_J,\ e_{1,1} = s_{IJ} - s_I - s_J.$ 
These formulas can be also deduced from a more general result of Teissier \cite[Corollary 8.7]{Te}. \par

If $(R,\mm)$ is a two-dimensional regular local ring, Verma \cite[Corollary 3.5]{Ve} showed that 
\begin{align*}
\ell(R/I^mJ^n)  =  &\; e(I){m \choose 2} + e(J){n \choose 2} + \big(\ell(R/IJ) - \ell(R/I) - \ell(R/I)\big)mn\\
& + \ell(R/I)m + \ell(R/J)n
\end{align*}
for all $m,n \ge 0$, where $e(I)$ and $e(J)$ denote the multiplicities of $I$ and $J$. To  prove this result he
used the theory of joint reductions of complete ideals and Lipman's formula $e_{1,1} = \ell(R/IJ) - \ell(R/I) - \ell(R/I)$ \cite{Li}.
By Teissier's result one has $e(I) = 2s_I$ and $e(J) = 2s_J$. If we use Lemma \ref{colength} and Lemma \ref{product number} 
to compute  $\ell(R/I),\ell(R/J)$ and $\ell(R/IJ)$ we can also recover Theorem \ref{Bhattacharya} from Verma's result. 
\par

As we have seen in the previous sections, the coefficients of the Bhattacharya polynomial in Theorem \ref{Bhattacharya} can be written down explicitly if we know the vertices of the Newton polyhedra of $I,J$ or the decompositions of $I,J$ as products of integral closures of complete intersections. To see that we consider the case $J = \mm$.

\begin{Theorem}   \label{maximal ideal}
Let $I = J_1\cdots J_s$, where $J_i = \overline{(x^{c_i},y^{d_i})}$ for some positive integers $c_i, d_i$, $i = 1,...,r$.
Assume that $d_1/c_1 \ge \cdots \ge d_r/c_r$. Let $s = \max\{i|\  d_i/c_i \ge 1\}$, 
where $s = 0$ if $d_1/c_1 < 1$. For all $m,n \ge 0$,
\begin{align*}
\ell(R/I^m\mm^n) = &\ \big(\sum_{i=1}^r\frac{c_id_i}{2} + \sum_{1 \le i < j \le r}c_id_j\big)m^2 + \frac{n^2}{2} + \big(\sum_{i=1}^sc_i + \sum_{j=s+1}^rd_j\big)mn \\
& + \frac{1}{2} \sum_{i=1}^r\big(c_i +d_i-\gcd(c_i,d_i)\big)m + \frac{n}{2}.
\end{align*}
\end{Theorem}

\begin{proof}
By Lemma \ref{area}, $s_I$ is the coefficient of $m^2$ in the right hand side of  the above formula.
It is clear that  $s_\mm = 1/2$. Since  $d_1/c_1 \ge \cdots \ge d_s/c_s \ge 1/1 > d_{s+1}/c_{s+1} \ge \cdots \ge d_r/c_r$, applying Lemma \ref{area} to the ideal $I \mm$ we get
$$s_{I\mm} = \sum_{i=1}^r\frac{c_id_i}{2} + \frac{1}{2} + \sum_{1 \le i < j \le r}c_id_j + \sum_{i=1}^sc_i + \sum_{j=s+1}^rd_j = s_I + \frac{1}{2} + \sum_{i=1}^sc_i + \sum_{j=s+1}^rd_j.$$ 
By definition, $a_I = \sum_{i=1}^r c_i, b_I = \sum_{i=1}^r d_i$, 
and by Lemma \ref{boundary}, $l_I = \sum_{i=1}^r\gcd(c_i,d_i) + 1.$ 
Moreover,  $a_\mm = b_\mm = 1$, $l_\mm = 2$. Putting these data into Theorem \ref{Bhattacharya} we obtain the assertion.
\end{proof} 

Theorem \ref{maximal ideal} contains explicit formulas for the Hilbert-function of the associated graded ring 
$G(I) = \oplus_{m \ge 0}I^m/I^{m+1}$ and the fiber ring $F(I) = \oplus_{m \ge 0}I^m/\mm I^m$.

\begin{Corollary}   \label{associated}
Let $I = J_1\cdots J_s$, where $J_i = \overline{(x^{c_i},y^{d_i})}$ for some positive integers $c_i, d_i$, $i = 1,...,r$.
Assume that $d_1/c_1\ge \cdots \ge d_r/c_r$. For all $m\ge 0$, 
$$
\ell(I^m/I^{m+1}) = \big(\sum_{i=1}^r\frac{c_id_i}{2} + \sum_{1 \le i < j \le r}c_id_j\big)(2m+1)  
+ \frac{1}{2} \sum_{i=1}^r\big(c_i +d_i-\gcd(c_i,d_i) \big).
$$
\end{Corollary}

\begin{proof}
Since $\ell(I^m/I^{m+1}) =  \ell(R/I^{m+1}) - \ell(R/I^m)$, 
the assertion follows from Theorem \ref{maximal ideal} by putting $n = 0$.
\end{proof} 

\begin{Corollary}   \label{fiber}
Let $I = J_1\cdots J_s$, where $J_i = \overline{(x^{c_i},y^{d_i})}$ for some positive integers $c_i, d_i$, $i = 1,...,r$.
Assume that $d_1/c_1 \ge \cdots \ge d_r/c_r$. Let $s = \max\{i|\  d_i/c_i \ge 1\}$,
where $s = 0$ if $d_1/c_1 < 1$.  For all $m\ge 0$, 
$$\ell(I^m/\mm I^m) = \big(\sum_{i=1}^sc_i + \sum_{j=s+1}^rd_j\big)m + 1.$$
\end{Corollary}

\begin{proof}
Since $\ell(I^m/\mm I^m) = \ell (R/\mm I^m) - \ell(R/I^m)$, 
the assertion follows from Theorem \ref{maximal ideal} by putting $n = 0,1$.
\end{proof} 

In particular, the case $m=1$ of Corollary \ref{fiber} yields 
the following formula for the minimal number of generators $v(I)$ of the ideal $I$.

\begin{Corollary}   \label{generator}
Let $I$ be a complete monomial ideal as above. Then 
$$v(I) =  \sum_{i=1}^sc_i + \sum_{j=s+1}^rd_j + 1.$$
\end{Corollary}

We can also apply Theorem \ref{Bhattacharya} to study the function $\ell(I^mJ^n/I^{m+1}J^n)$ 
for an $\mm$-primary  complete monomial ideal $I$ and an arbitrary complete monomial ideal $J$ in $R$. 
This follows from the fact that $J = fJ'$ for a monomial $f$ and an $\mm$-primary complete monomial ideal $J'$. Hence 
$$\ell(I^mJ^n/I^{m+1}J^n) = \ell(I^m(J')^n/I^{m+1}(J')^n) = \ell(R/I^{m+1}(J')^n) - \ell(R/I^m(J')^n).$$


\begin{thebibliography}{1}

\bibitem{Bh} P.B. Bhattacharya, The Hilbert function of two ideals, Math. Proc. Cambridge Philos. Soc. 53 (1957), 568--575.

\bibitem{Cr}  V. Crispin Quinonez, Integrally closed monomial ideals and powers of ideals,
Research Reports in Mathematics No. 7, Department of Mathematics, Stockholm University, 2002.

\bibitem{Co} D. Cox, J. Little and D. O'Shea, Using Algebraic Geometry, Springer, New York, 1998.

\bibitem {Gr} B. Gr\"unbaum and G. C. Shephard, Pick’s Theorem, Amer. Math. Monthly 100 (1993), 150--161.

\bibitem{HuS} C. Huneke and I. Swanson, Integral closure of Ideals, Rings and Modules, Cambridge University Press, 2006.

\bibitem{KaV} D. Katz and J. Verma, Extended Rees algebras and mixed multiplicities, Math. Z. 202 (1989), 111--128.

\bibitem {Li} J. Lipman, On complete ideals in regular local rings, in: Commutative Algebra and
Algebraic Geometry, vol. I, Kinokuniya, Tokyo, 1988, 203--231.

\bibitem{Te} B. Teissier, Monomial ideals, binomial ideals, polynomial ideals, in: 
Trends in commutative algebra, Math. Sci. Res. Inst. Publ., vol. 51, Cambridge University Press, 2004, 211--246.

\bibitem{TrV} N.V. Trung and J. Verma, Mixed multiplicities of ideals versus mixed volumes of polytopes,
Trans. Amer. Math. Soc. 359 (2007), 4711--4727.

\bibitem{TrV2} N.V. Trung and J. Verma, Hilbert functions of multigraded algebras, mixed multiplicities of ideals and their applications. J. Commut. Algebra 2 (2010), 515--565.

\bibitem{Ve}  J. Verma, Joint reductions of complete ideals, Nagoya Math. J. 118 (1990), 155--163.

\bibitem{Va}  W. Vasconcelos, Integral closure: Rees algebras, Multiplicities, Algorithm, Springer, 2005.

\bibitem{Za} O. Zariski and P. Samuel, Commutative Algebra, vol. 2, van Nostrand, Princeton, 1960.
 
\end{thebibliography}
\end{document}